\documentclass[12pt, draftcls, onecolumn]{IEEEtran}

\usepackage[colorlinks=false,urlcolor=blue,citecolor=blue,linkcolor=blue,bookmarks=true,bookmarksopen=false,pdftitle=created by dvipdf,pdfcreator=NM,pdfauthor=Nayyar,pdfsubject=mor.sty:6.29.04]{hyperref}
\usepackage{amsmath,amsthm,amstext,amsfonts,amssymb,mathrsfs}
\usepackage{graphicx,subfigure,color,algorithm,algorithmic,fullpage,setspace}

\usepackage{accents}
\newlength{\dhatheight}
\newcommand{\doublehat}[1]{%
    \settoheight{\dhatheight}{\ensuremath{\hat{#1}}}%
    \addtolength{\dhatheight}{-0.25ex}%
    \hat{\vphantom{\rule{1pt}{\dhatheight}}%
    \smash{\hat{#1}}}}

\usepackage{personal}

\IEEEoverridecommandlockouts
\newtheorem{remark}{Remark}

\title{Optimal Decentralized Control with Asymmetric One-Step Delayed Information Sharing}
\author{Naumaan Nayyar, Dileep Kalathil and Rahul Jain\thanks{The authors are with the Department of Electrical Engineering, University of Southern California, Los Angeles, CA 90089. Emails: {\tt (nnayyar,manisser,rahul.jain)@usc.edu}.}\thanks{This research was supported by AFOSR grant FA9550-10-1-0307, the NSF CAREER Award CNS-0954116 and the ONR Young Investigator Award N000141210766. Parts of this paper have been presented in ACC 2014, without all proof details and only the $(1,\infty)$ case. Theorems 4 and 6 are new, which now fully solve the problem.}}

\begin{document}

\maketitle

\begin{abstract}
We consider optimal control of decentralized LQG problems for plants having nested subsystems controlled by two players with asymmetric information sharing patterns between them. In the main scenario, the players are assumed to have a unidirectional error-free, unlimited-rate communication channel with a unit delay in one direction and no communication in the other. A second model, presented for completeness, considers a channel with no delay in one direction and a unit delay in the other. Delayed information sharing patterns do not, in general, admit linear optimal control laws and are thus difficult to control optimally. However, in these scenarios, we show that the problems have a partially nested information structure, and thus linear optimal control laws exist. Summary statistics are identified and analytical solutions to the optimal control laws are derived. State and output feedback cases are solved for both scenarios.
\end{abstract}

\vspace*{-0.20in}

\section{Introduction}

Recently, many problems of decentralized `non-classical' control have been looked at in the context of practical systems~\cite{RoLa06, LaChDa04, MoJa08}. Examples include cyberphysical systems, formation flight, and other networked control systems wherein multiple agents try to achieve a common objective in a decentralized manner. Such situations arise, for example, due to controllers not having access to the same information. One possible reason is network delays, and a consequent time-lag to communicate observations to other controllers.

These problems were first formulated by Marschak in the 1950s \cite{Ma5x} as team decision problems, and further studied by Radner \cite{Ra62}, though in such problems communication between the controllers was usually ignored. 
In a celebrated paper \cite{Wi68}, Witsenhausen  showed that even for seemingly simple systems with just two players, non-linear controllers could outperform any linear controller. Witsenhausen also consolidated and conjectured results on separation of estimation and control in decentralized problems in~\cite{Wi71}. However, the structure of decentralized optimal controllers for LQG systems with time-delays has been hard to identify. Indeed, in \cite{VaWa78} it was proven that the separation principle that was conjectured for delayed systems does not hold for a system having a delay of more than one timestep. The more general delayed information sharing pattern was only recently solved by Nayyar, et al. in \cite{NaMaTe11}. It is an open problem to actually compute the optimal decentralized control law for the $n$-player setting even using such a structural result.

However, not all results for decentralized control problems with delayed information sharing patterns have been negative. Building on results of Radner on team decision theory, Ho and Chu \cite{HoCh72} showed that for a unit-delay information sharing pattern, the optimal controller is linear. This was used by Kurtaran and Sivan \cite{KuSi74}, and others~\cite{SaAt74,Yo75} to derive optimal controllers for the finite-horizon case. It is also known that for a class of problems having partially nested structures, person-by-person optimality implies global optimality~\cite{MaMaRo12}. These approaches do not extend to multi-unit delayed sharing patterns as they do not possess partially nested structures.

A criterion for determining convexity of optimal control problems called \textit{funnel causality}, was presented in~\cite{BaVo05}. Recently, another characterization, \textit{quadratic invariance}, was discovered under which optimal decentralized controls laws can be solved by convex programming \cite{RoLa06}. This has led to a resurgence of interest in decentralized optimal control problems, which since the 1970s had been assumed to be intractable. Subsequently, in a series of papers, Lall, Lessard and others have computed the optimal control laws for a suite of decentralized nested LQG systems \cite{SwLa10,SwLa11,LeNa13}. However, none of these papers consider the effect of communication delays. More general networked structures that dealt with state-feedback delayed information sharing in networked control with sparsity were also looked at in~\cite{LaLe12, LaLe14}, wherein recursive solutions for optimal control laws were derived. Multi-player control using output feedback without delay was considered in~\cite{NaLe14}. Time varying delays in 2-player problems, restricted to state feedback, have known explicit solutions~\cite{MaLaDo14}.

For a subclass of quadratic invariant problems known as \textit{poset-causal} problems, Parrilo and Shah \cite{ShPa10} showed that the computation of the optimal law becomes easier by lending itself to decomposition into subproblems. Solutions to certain cases of the state-feedback and output-feedback delayed information sharing problem have also been presented by Lamperski et al.~\cite{LaDo12} where a networked graph structure of strongly connected controllers is considered, with constraints on system dynamics. 
In this work, we restrict our attention to two-player systems. A summary of results pertaining to these systems is given in Table \ref{tab:lit-survey}.

\vspace*{-0.05in}
\begin{table}[htb]
\centering
%
\begin{tabular}{|@{} c @{}|@{} c @{}|@{} c @{}| c @{}|}
\hline
\textbf{$(d_{12},~d_{21})$} & \textbf{Literature} & \textbf{Comments}\\ \hline
$(0,0)$ & Classical & no plant restrictions\\ \hline
$(1,1)$ & \cite{KuSi74},\cite{SaAt74},\cite{Yo75} & no plant restrictions\\ \hline
$(0 \text{ or } 1, 0 \text{ or } 1)$ & \cite{LaDo12,LaDo13} & b.d B matrix, inf. horizon\\ \hline
$(0,\infty)$ & \cite{SwLa10,SwLa11,LeNa13} & l.b.t matrix\\ \hline
$(1,\infty)$ & \cite{LaLe12} & l.b.t matrix, state f/b, u.c\\ \hline
$(\infty,\infty)$ & \cite{Le12} & b.d. dynamics, state f/b \\ \hline
$(1,0)$ & here & no plant restrictions\\ \hline
$(1,\infty)$ & here & l.b.t matrix\\ \hline
\end{tabular}
\caption{Summary of results for some information sharing patterns with two-players.}
\label{tab:lit-survey}
\end{table}
\vspace{-5mm}
In Table \ref{tab:lit-survey}, $d_{12}$ is the delay in information transmission from player 1 to player 2, and vice versa. 'u.c', 'b.d' and 'l.b.t' refer to uncoupled noise, block diagonal and lower block triangular matrices respectively. Note that, in ~\cite{LaLe12}, noise vectors are restricted to be independent between subsystems, unlike in our present work.

In this paper, we consider two asymmetric scenarios, $(1,0)$ and $(1,\infty)$ information sharing patterns between two players in an LQG system.
Both scenarios have a partially nested information structure and also satisfy quadratic invariance. They are not poset-causal and do not lend themselves to easy decompositions. We derive optimal control laws in both cases. This results in Riccati-type iterations for computing the gain matrix of the optimal control law which is linear. This paper extends our work presented in ACC 2014~\cite{NaKaJa14} to derive explicit optimal controllers for the output feedback case.

\vspace*{-0.14in}
\section{The $(1,\infty)$ information sharing pattern}\label{sec:problem}

\vspace*{-0.03in}
\subsection*{State-feedback case}

\vspace*{-0.03in}
We consider a coupled two-player discrete linear time-invariant system with nested structure. Player 1's actions affect block 1 of the plant whereas player 2's actions affect both blocks 1 and 2. The system dynamics are,
\vspace*{-0.05in}
\begin{align}
\label{eq:nestedsystem}
\begin{bmatrix} x_1(t+1)\\ x_2(t+1) \end{bmatrix} 
=
&\begin{bmatrix} A_{11} & 0\\ A_{21} & A_{22} \end{bmatrix} 
\begin{bmatrix} x_1(t)\\ x_2(t) \end{bmatrix} +\nonumber\\ &\begin{bmatrix} B_{11} & 0\\ B_{21} & B_{22} \end{bmatrix} 
\begin{bmatrix} u_1(t)\\ u_2(t) \end{bmatrix} + \begin{bmatrix} v_1(t)\\ v_2(t) \end{bmatrix},
\end{align} 
for $t \in \{0,1,\cdots,N-1\}$.

We will use the notation $x(t+1) = A x(t) + B u(t) + v(t)$ where definition is clear from context. The initial state $x(0)$ is assumed to be zero-mean Gaussian, independent of the system noise $v(t)$, and $v(t)$ is zero-mean Gaussian with covariance $V$, which is independent across time $t$.

In the formulation above, restricting the plant dynamics to be lower block triangular ensures that optimal control laws exist that are linear in the observations. Without this assumption, the system fails the only known sufficiency test for existence of linear optimal laws for two players~\cite{HoCh72}.

The two players have a common (team) objective to control the system to minimize,
\begin{equation}\label{eq:costfn}
\bbE[\sum_{t=0}^{N-1}(x(t)'Qx(t) + u(t)'Ru(t)) + x(N)'Sx(N)],
\end{equation}
where $Q$, $R$ and $S$ are positive definite matrices. At each instant, the control action $u_i(t)$ taken by each player can only depend on the \textit{information structure}, the information available to them denoted by $\sF(t)=(\sF_1(t),\sF_2(t))$. The information structure for the problem is,
\begin{align}\label{eq:infostructure}
\sF_1(t) = \{&x_1(0:t),u_1(0:t-1)\},\nonumber\\
\sF_2(t) = \{&x_1(0:t-1),u_1(0:t-1);\nonumber\\
&~~ x_2(0:t),u_2(0:t-1)\}.
\end{align}

The set of all control laws is characterized by $u_i(t)=f(\sF_i(t), t)$ for some time-varying function $f$. In the next subsection, we show that the optimal law is linear.


\subsubsection{Linearity of the optimal control law}\label{sec:linearity}


It was shown by Varaiya and Walrand~\cite{VaWa78} that the optimal decentralized control law with general delayed information sharing between players may not be linear. However, Ho and Chu~\cite{HoCh72} had established earlier that to show the existence of a linear optimal control law, it is sufficient to prove that the LQG problem has a \textit{partially nested information structure}~\cite[Theorem 2]{HoCh72}.

Formally, if player $i$'s actions at time $t_1$ affect player $j$'s information at time $t_2$, then a partially nested information structure is defined to have $\sF_i(t_1) \subseteq \sF_j(t_2)$. Intuitively, this can be described as an information pattern that allows communication of all information used by a player to make a decision to all other players whose system dynamics are affected by that decision. We now show that the $(1,\infty)$ pattern has a partially nested structure.

\begin{proposition}\label{lem:linearity}
The problem~\eqref{eq:nestedsystem}-\eqref{eq:costfn} with information structure~\eqref{eq:infostructure} has a linear optimal control law.
\end{proposition}
\begin{proof}
Observe that the information structure defined in~\eqref{eq:infostructure} can be simplified as $\sF_1(t) = \{x_1(0:t)\}$ and $\sF_2(t) = \{x_1(0:t-1),x_2(0:t) \}$ as the observations can be used to determine inputs. Because of the nested system structure, viz. $A_{12} = B_{12} = 0$, the propagation of inputs in the plant dynamics is as shown in Fig~\ref{fig:optonedelayonecomm}. Additionally, at any time instant $t$, $\sF_i(t) \subset \sF_i(t+\tau), \tau \geq 0, i = 1,2$ and $\sF_1(t) \subset \sF_2(t+\tau), \tau \geq 1$. This implies that the information communication diagram is also the same as in Figure~\ref{fig:optonedelayonecomm}.
\begin{figure}[tbh]
\centering{\includegraphics[scale=0.8]{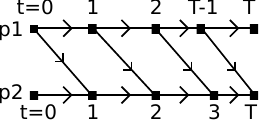}}
\caption{Information communication in $(1,\infty)$-delayed sharing pattern in two-player D-LTI systems with nested structures}
\label{fig:optonedelayonecomm}
\end{figure}

Thus, the system described with its information structure has a partially nested structure \cite[Theorem 2]{HoCh72}, and consequently, a linear optimal control law exists.
\end{proof}

Denote by $\sH_i(t) = ( x_i(0:t-1), u_i(0:t-1))$, the history of observations for block $i$. Using linearity of the optimal law, it follows that $u(t)=[u_1(t),~u_2(t)]'=$
\begin{equation}\label{eq:optlinlaw}
\begin{bmatrix}
F^*_{11}(t)x_1(t) + L_1(\sH_1(t)) \\
F^*_{22}(t)x_2(t) + L_2(\sH_1(t),\sH_2(t))
\end{bmatrix},
\end{equation}
where $L_i(\cdot)$ are linear, possibly time-varying, functions. However, the arguments of $L_i$, namely the output observations, increase with time. So, obtaining a closed form expression for the optimal control law is intractable in the current form, necessitating use of summary statistics.

\subsubsection{Summary Statistic to Compute Optimal Control Law}\label{sec:optimality}

\begin{table}[b]
\centering
\begin{tabular}{| c | c |} \hline
  $\sH_i(t)$ & $\{ x_i(0:t-1), u_i(0:t-1)\}$\\ \hline
  $\sF_1(t)$ & $\{ x_1(t),\sH_1(t)\}$\\ \hline
  $\sF_2(t)$ & $\{ x_2(t),\sH_1(t), \sH_2(t)\}$\\ \hline
  $\hat{x}(t)$ & $\bbE[x(t) | \sH_1(t)]$\\ \hline
  $\doublehat{x}(t)$ & $\bbE[x(t) | \sH_1(t),\sH_2(t)]$\\ \hline
  $\phi(t)$ & $x_1(t) - \hat{x}_1(t)$\\ \hline
  $\overline{\phi}(t)$ & $x(t) - \doublehat{x}(t)$\\ \hline
\end{tabular}
\caption{A glossary of notations used in the following sections}
\label{tab:glossary-terms}
\end{table}
We now show that the estimate of the current state given the history of observations of the players is a summary statistic that characterizes the optimal control law. The notation used is summarized in Table \ref{tab:glossary-terms}.

Note that  $\hat{x}(t)$ and $\doublehat{x}(t)$ are estimates of the current states by Players $1$ and $2$ respectively.
\begin{table}[t]
\centering
\begin{tabular}{| c | c |} \hline
  $\overline{T}(t)$ & $\bbE[(x(t) - \doublehat{x}(t))(x(t) - \doublehat{x}(t))]$ \\ \hline
  $T(t)$ &  $\bbE[(x(t) - \hat{x}(t))(x_1(t) - \hat{x}_1(t))]$ \\ \hline
  $\overline{R}_{x}(t)$ & $\overline{T}(t)$ \\ \hline
 $\overline{R}_{xx}(t)$ & $A\overline{T}(t)$  \\ \hline
  $R_{x_1}(t)$ & $T_1(t)$ \\ \hline
  $R_{xx_1(t)}$ & $AT(t)$\\ \hline
  $\overline{\phi}(t)$,$\phi(t)$ & $x(t) - \doublehat{x}(t)$,$x_1(t) - \hat{x}_1(t)$\\ \hline
  $\overline{\sR}(t)$ & $\overline{R}_{xx}(t)\overline{R}_{x}^{-1}(t)$\\ \hline
  $\sR(t)$ & $R_{xx_1}(t)R_{x_1}^{-1}(t)$\\ \hline
  $\hat{u}_2(t)$ & $\bbE[u_2(t) | \sH_1(t)]$\\ \hline
\end{tabular}
\caption{A list of notations used in Lemma 1}
\label{tab:lemma-1-notation}
\end{table}

\begin{lemma}\label{lem:estimatordynamics}
For the nested system~\eqref{eq:nestedsystem}, the estimator dynamics for $\hat{x}(t)$ and $\doublehat{x}(t)$ are given by
\begin{equation}
\hat{x}(t+1) = A \hat{x}(t) +  B\begin{bmatrix} u_1(t)\\
\hat{u}_2(t)
\end{bmatrix} + \sR(t)\phi(t),
\label{eq:estdyn1}
\end{equation}
\vspace*{-0.1in}
\begin{equation}
\doublehat{x}(t+1) = A \doublehat{x}(t) +  B\begin{bmatrix} u_1(t)\\
u_2(t)
\end{bmatrix} + \overline{\sR}(t) \overline{\phi}(t),
\label{eq:estdyn2}
\end{equation}

with $\hat{x}(0) = \doublehat{x}(0) = \bbE[x(0)]$. $\overline{\sR}(t) = A$. $\phi(t)$ and $\overline{\phi}(t)$ are zero-mean, uncorrelated Gaussian random vectors and $\bbE[\phi(t)\phi(t)'] = T_1(t)$ and $\bbE[\overline{\phi}(t)\overline{\phi}(t)'] = \overline{T}(t)$. At each time step, $\phi(t)$ is uncorrelated with $\sH_1(t)$, and so $\bbE[\phi(t)x_1(\t)'] = \mathbf{0}$ for $\tau < t$. $\overline{\phi}(t)$ is uncorrelated with $\sH_1(t),~\sH_2(t)$, and so $\bbE[\overline{\phi}(t)x(\t)'] = \mathbf{0}$ for $\tau < t$.
\end{lemma}
\begin{proof}
The proof is in the Appendix.
\end{proof}

\vspace*{-0.17in}
Before showing that $\hat{x}_1(t)$ and $\doublehat{x}_2(t)$ are summary statistics for the history of state observations, we first obtain the optimal decentralized control law for the LQG problem~\eqref{eq:nestedsystem}-\eqref{eq:costfn} with a nested information structure having no communication delay - i.e., $(0,\infty)$. The information structure here is described by,
\begin{align}\label{eq:infostructureLeNa}
z_1(t) = \{&x_1(0:t),u_1(0:t-1)\},\\
z_2(t) = \{&x_1(0:t),u_1(0:t-1); x_2(0:t),u_2(0:t-1)\}.\nonumber 
\end{align}

Proposition~\ref{lem:lall} considers the $(0,\infty)$ state-feedback problem, already solved for output-feedback in~\cite{LeNa13}:
\begin{proposition}\label{lem:lall}
The optimal control law for the LQG problem~\eqref{eq:nestedsystem}-\eqref{eq:costfn} with information structure~\eqref{eq:infostructureLeNa} is,
\begin{equation*}
\begin{bmatrix}
u_1^*(t)\\
u_2^*(t)
\end{bmatrix} = \begin{bmatrix}
K_{11}(t) & K_{12}(t) & \mathbf{0} \\
K_{21}(t) & K_{22}(t) & J(t)
\end{bmatrix} \begin{bmatrix}
x_1(t)\\
\tilde{x}_2(t)\\
x_2(t) - \tilde{x}_2(t)
\end{bmatrix},
\end{equation*}
\begin{align}
\label{eq:LeNadynamics}
\tilde{x}_2(t+1) &= A_{21}x_1(t) + A_{22}\tilde{x}_2(t) + B_{21} u_1(t) + B_{22}\tilde{u}_2(t),\nonumber\\
\tilde{x}_2(0) &= 0,
\end{align}
where $\tilde{u}_2(t) = \begin{bmatrix}K_{21}(t) & K_{22}(t)\end{bmatrix} \begin{bmatrix} x_1(t)\\\tilde{x}_2(t) \end{bmatrix}$.
$K(t)$ is obtained from the backward recursion,
\begin{align*}
P(T) &= S,\\
P(t) &= A'P(t+1)A + A'P(t+1)BK(t) + Q,\\
K(t) &= -(B'P(t+1)B + R)^{-1}B'P(t+1)A,
\end{align*}
and $J(t)$ is obtained from the backward recursion,

\begin{align*}
&\tilde{P}(T) = S_{22},\\
&\tilde{P}(t) = A_{22}'\tilde{P}(t+1)A_{22} + A_{22}'\tilde{P}(t+1)B_{22}J(t) + Q_{22},\\
&J(t) = -(B_{22}'\tilde{P}(t+1)B_{22} + R_{22})^{-1}B_{22}'\tilde{P}(t+1)A_{22}.
\end{align*}
\end{proposition}
\begin{proof}
The proof follows by use of Theorems 10 and 12 of~\cite{LeNa13}, and replacing the information structure with the state-feedback structure given in~\eqref{eq:infostructureLeNa}. The conditional expectations reduce to $\textbf{z}(t) := \bbE[x(t) | z_2(t)] = x(t)$ and $\hat{\textbf{z}}(t) := \bbE[x(t) | z_1(t)] = \begin{bmatrix} x_1(t)\\ \tilde{x}_2(t)\end{bmatrix}$, where $\tilde{x}_2(t) := \bbE[x_2(t) | z_1(t)]$. The optimal control law is,
\begin{align*}
u(t) = K(t)\hat{\textbf{z}}(t) + \hat{K}(t)(\textbf{z}(t) - \hat{\textbf{z}}(t)),
\end{align*}
where $\hat{K}(t) = \begin{bmatrix}
0 & 0 \\
\hat{K}_{21} & J(t).
\end{bmatrix}$

Substituting $\textbf{z}(t)$ and $\hat{\textbf{z}}(t)$, the control law reduces to,
\begin{align*}
u(t) &= K(t)\begin{bmatrix}
x_1(t)\\
\tilde{x}_2(t)
\end{bmatrix} + \hat{K}(t)\begin{bmatrix}
0\\
x_2(t) - \tilde{x}_2(t)
\end{bmatrix}\\
     &= \begin{bmatrix}
K_{11}(t) & K_{12}(t) & \mathbf{0} \\
K_{21}(t) & K_{22}(t) & J(t)
\end{bmatrix} \begin{bmatrix}
x_1(t)\\
\tilde{x}_2(t)\\
x_2(t) - \tilde{x}_2(t)
\end{bmatrix}.
\end{align*}

Consequently, we also have, $\tilde{u}_2(t) := \bbE[u_2(t) | z_1(t)] = \begin{bmatrix}K_{21}(t) & K_{22}(t)\end{bmatrix} \begin{bmatrix} x_1(t)\\x_2(t) \end{bmatrix}$, which when substituted in~\eqref{eq:nestedsystem}, gives ~\eqref{eq:LeNadynamics}.
\end{proof}
We now establish the sufficiency of $\hat{x}_1(t)$ and $\doublehat{x}_2(t)$.
\begin{theorem}\label{thm:optcontrollerform}
For the nested system~\eqref{eq:nestedsystem} with the information structure~\eqref{eq:infostructure} and objective function~\eqref{eq:costfn}, the optimal control law is given by
\begin{equation}
u(t)  =  F^*(t)\begin{bmatrix}
x_1(t) - \hat{x}_1(t)\\
x_2(t) - \doublehat{x}_2(t)
\end{bmatrix} + \tilde{K}(t),
\label{eq:optcontroller}
\end{equation}
where $F^*(t) = \begin{bmatrix}
F^*_{11}(t) &  0\\
0 & F^*_{22}(t)
\end{bmatrix}$ and $\tilde{K}(t)$ is given by
\begin{align*}
\begin{bmatrix}
K_{11}(t) & K_{12}(t) & \mathbf{0} \\
K_{21}(t) & K_{22}(t) & J(t)
\end{bmatrix} \begin{bmatrix}
\hat{x}_1(t)\\
\hat{x}_2(t)\\
\doublehat{x}_2(t) - \hat{x}_2(t)
\end{bmatrix},
\end{align*}
where $K(t)$ and $J(t)$ are as in Proposition \ref{lem:lall}.
\end{theorem}
The proof is in the Appendix. Note that the classical separation principle does not hold as the statistics $\hat{x}(t)$ and $\doublehat{x}(t)$ cannot be used to compute the optimal law.

\subsubsection{Deriving the Optimal Gain Matrix}
\label{sec:optgain}
To characterize the optimal gain matrix, $F^*(t)$, the stochastic optimization problem is converted into a deterministic matrix optimization, and solved analytically.

\begin{theorem}\label{thm:optgainmatrix}
The matrix $F^*(t)$ is characterized by the linear equations,

\begin{align}
\Big[R + \tilde{B}'M(t+1)\tilde{B}\Big]_{11}F_{11}^*(t)\overline{T}_{11}(t) \nonumber \\
+ \Big[R + \tilde{B}'M(t+1)\tilde{B}\Big]_{12}F_{22}^*(t)\overline{T}_{21}(t) \nonumber\\
 = -\Big[\tilde{B}'M(t+1)V(t)'\Big]_{11} \\
\Big[R + \tilde{B}'M(t+1)\tilde{B}\Big]_{21}F_{11}^*(t)\overline{T}_{12}(t) \nonumber \\
+ \Big[R + \tilde{B}'M(t+1)\tilde{B}\Big]_{22}F_{22}^*(t)\overline{T}_{22}(t) \nonumber\\
 = -\Big[\tilde{B}'M(t+1)V(t)'\Big]_{22},
\label{eq:optgainmatrix}
\end{align}
where $[A]_{ij} = E_i'AE_j$, $E_i$ and $E_j$ being suitable partitions of the identity matrix $I = [E_i~E_j]$.

$\tilde{A} = \begin{bmatrix}
A_{11} & 0 & 0\\
A_{21} & A_{22} & 0\\
0 & 0 & A_{22}
\end{bmatrix}$, $G(t) = \begin{bmatrix}
BH(t)\\
[0~0~B_{22}J(t)]
\end{bmatrix}$, $\tilde{B} = \begin{bmatrix}
B_{11} & 0\\
B_{21} & B_{22}\\
0 & B_{22}
\end{bmatrix}$, $H(t) = \begin{bmatrix}
K_{11}(t) & K_{12}(t) & -K_{12}(t) \\
K_{21}(t) & K_{22}(t) & J(t) - K_{22}(t)
\end{bmatrix}$, $\tilde{Q} = \begin{bmatrix}
Q & 0\\
0 & 0
\end{bmatrix}$, $\tilde{S} = \begin{bmatrix}
S & 0\\
0 & 0
\end{bmatrix}$, $V(t) = \mathbf{E}[\overline{\phi}(t)n_1(t)']$, and $n_1(t) := \begin{bmatrix}
(\sR(t)\overline{\phi}_1(t))_1\\
B_{21}F_{11}(t)\overline{\phi}_1(t) + (\overline{\sR}(t)\overline{\phi}(t))_2\\
+ (\overline{\sR}(t)\overline{\phi}(t))_2 - (\sR(t)\overline{\phi}_1(t))_2
\end{bmatrix}$.
\end{theorem}
The proof is in the Appendix.

\vspace*{-0.15in}
\subsection*{Output-feedback case}
\label{sec:outputfeedback}
The approach extends to the output-feedback case where players' observations are corrupted by noise. The system dynamics are described by the same lower block triangular matrices as,
\begin{align}\label{eq:nestedsystemO}
\begin{bmatrix}
x_1(t+1)\\
x_2(t+1)
\end{bmatrix} &= \begin{bmatrix}
A_{11} & 0\\
A_{21} & A_{22}
\end{bmatrix} \begin{bmatrix}
x_1(t)\\
x_2(t)
\end{bmatrix} \nonumber\\
&+ \begin{bmatrix}
B_{11} & 0\\
B_{21} & B_{22}
\end{bmatrix} \begin{bmatrix} u_1(t)\\
u_2(t)
\end{bmatrix} + \begin{bmatrix}
v_1(t)\\
v_2(t)
\end{bmatrix},\nonumber\\
\begin{bmatrix}
y_1(t)\\
y_2(t)
\end{bmatrix} &= \begin{bmatrix}
C_{11} & 0\\
C_{21} & C_{22}
\end{bmatrix} \begin{bmatrix}
x_1(t)\\
x_2(t)
\end{bmatrix} + \begin{bmatrix}
w_1(t)\\
w_2(t)
\end{bmatrix}.
\end{align}
Here, $v(t) := \begin{bmatrix}v_1(t)\\
v_2(t)\end{bmatrix}$ and $w(t):= \begin{bmatrix}w_1(t)\\
w_2(t)\end{bmatrix}$ are zero-mean, Gaussian random vectors with covariance matrices $V$ and $W$ respectively that are independent across time and independent of initial system state $x(0)$ and of each other. $x(0)$ is zero-mean, Gaussian with known mean and covariance.

The information structure for the problem is,
\begin{align}\label{eq:infostructureO}
\sF_1(t) =\{&y_1(0:t),u_1(0:t-1)\}, \nonumber \\
\sF_2(t) =\{&y_1(0:t-1),u_1(0:t-1);\nonumber\\
&~~ y_2(0:t),u_2(0:t-1))\}.
\end{align}

Redefining the history of observations as, $\sH_i(t) = ( y_i(t-1),...,y_i(0),u_i(t-1),...,u_i(0))$ for block $i$, $i=1,2$, the information available to Player $1$ at time $t$ can also be expressed as $\sF_1(t) = \{ y_1(t),\sH_1(t)\}$, and that available to Player $2$ as $\sF_2(t) = \{ y_2(t),\sH_1(t),\sH_2(t)\}$. 

Given this information structure, the players' objective is to find the control law $u^*_i(t)$ as a function of $\sF_i(t)$ $(i=1,2)$ that minimizes the finite-time quadratic cost criterion~\eqref{eq:costfn}.  It can be seen from Proposition~\ref{lem:linearity} that the optimal control law for the formulated output-feedback problem is linear.

Further, defining estimators $\hat{x}(t) := \bbE[x(t)|\sH_1(t)]$ and $\doublehat{x}(t) := \bbE[x(t)|\sH_2(t)]$, the following result can be reconstructed from the state-feedback version.
\begin{lemma}\label{lem:estimatordynamicsO}
For the nested system~\eqref{eq:nestedsystemO}, the estimator dynamics for the redefined $\hat{x}(t)$ and $\doublehat{x}(t)$ are as follows.
\begin{equation}
\hat{x}(t+1) = A \hat{x}(t) +  B\begin{bmatrix} u_1(t)\\
\hat{u}_2(t)
\end{bmatrix} + \sR(t)\phi(t),
\label{eq:estdyn1O}
\end{equation}
\begin{equation}
\doublehat{x}(t+1) = A \doublehat{x}(t) +  B\begin{bmatrix} u_1(t)\\
u_2(t)
\end{bmatrix} + \overline{\sR}(t) \overline{\phi}(t)
\label{eq:estdyn2O}
\end{equation}
and $\hat{x}(0) = \doublehat{x}(0) = \bbE[x(0)]$, where
$\hat{u}_2(t) = \bbE[u_2(t) | \sH_1(t)]$. $\sR(t) = R_{xy_1}(t)R_{y_1}^{-1}(t)$, $\overline{\sR}(t) = \overline{R}_{xy}\overline{R}_{y}^{-1}$,and $\phi(t) = y_1(t) - C_{11}\hat{x}_1(t)$, $\overline{\phi}(t) = y(t) - C\doublehat{x}(t)$. $R_{xy_1} = AT(t)C_{11}'$, $R_{y_1} = C_{11}T_1(t)C_{11}' + W_{11}$, $\overline{R}_{xy} = A\overline{T}(t)C'$, $\overline{R}_{y} = C\overline{T}(t)C' + W$. $T(t) = \bbE[(x(t) - \hat{x}(t))(x_1(t) - \hat{x}_1(t))]$, $\overline{T}(t) = \bbE[(x(t) - \hat{x}(t))(x(t) - \hat{x}(t))']$.

Also, $\phi(t)$ and $\overline{\phi}(t)$ are zero-mean, uncorrelated Gaussian random vectors. $\bbE[\phi(t)\phi(t)'] = C_{11}T_1(t)C_{11}'+W_{11}$ and $\bbE[\overline{\phi}(t)\overline{\phi}(t)'] = C\overline{T}(t)C' + W$. Additionally, at each time step, $\phi(t)$ is uncorrelated with $\sH_1(t)$, and so $\bbE[\phi(t)y_1(\t)'] = \mathbf{0}$ for $\tau < t$. $\overline{\phi}(t)$ is uncorrelated with $\sH_1(t),~\sH_2(t)$, and so $\bbE[\overline{\phi}(t)y(\t)'] = \mathbf{0}$ for $\tau < t$.
\end{lemma}
\begin{proof}
Proof is similar to Lemma~\ref{lem:estimatordynamics} and is omitted.
\end{proof}

\begin{theorem}
For the nested system~\eqref{eq:nestedsystemO} with the information structure~\eqref{eq:infostructureO} and objective function~\eqref{eq:costfn}, the optimal decentralized control law is given by,
\begin{equation}
u(t)  =  F^*(t)\begin{bmatrix}
y_1(t) - \hat{x}_1(t)\\
y_2(t) - C_{21}\doublehat{x}_1(t) - C_{22}\doublehat{x}_2(t)
\end{bmatrix} + \tilde{K}(t),
\label{eq:optcontroller}
\end{equation}
where $F^*(t) = \begin{bmatrix}
F^*_{11}(t) &  0\\
0 & F^*_{22}(t)
\end{bmatrix}$ is the optimal gain matrix, and,
\begin{align*}
\tilde{K}(t) = &\begin{bmatrix}
K_{11}(t) & K_{12}(t) \\
K_{21}(t) & K_{22}(t)
\end{bmatrix} \begin{bmatrix}
\hat{x}_1(t)\\
\hat{x}_2(t)\\
\end{bmatrix} \\
& + \begin{bmatrix}
0 \\
\begin{bmatrix}
\hat{K}_{21}(t) & J(t)
\end{bmatrix} \begin{bmatrix}
\doublehat{x}(t) - \hat{x}(t)
\end{bmatrix}
\end{bmatrix}
,
\end{align*}
where $K(t)$ and $J(t)$ are as in Proposition \ref{lem:lall}, $\hat{K}_{21}$ is the optimal control gain for the second player's input given the optimal centralized controller only based on common information. ~\cite[Theorem 12]{LeNa13}.
\label{thm:optcontrollerformO}
\end{theorem}
The proof is in the appendix.

\vspace*{-0.05in}
\begin{remark}
Observe that the classical separation principle does not hold in the $(1,\infty)$-delayed sharing pattern~\cite{KuVa86}, i.e., $u_2(t)$ cannot be computed using just $\mathbf{E}[x(t) | \sF_2(t)]$.
\end{remark}

\begin{theorem}\label{thm:optgainmatrixO}

The optimal gain matrix $F^*(t)$ is given by the recursion,
\begin{align}
F^*(t) &= -\Big(R + \tilde{B}'M(t+1)\tilde{B}\Big)^{-1}\times\nonumber\\
&~~~~\Big(\Psi^*(t) + \tilde{B}'M(t+1)V(t)'\Big)\times\nonumber\\
&~~~~\Big(C\overline{T}(t)C' + W\Big)^{-1},\nonumber\\
M^*(t) &= \tilde{Q} + H(t)'RH(t)\nonumber\\
&~~~~ + (\tilde{A} + G(t))'M^*(t+1)(\tilde{A} + G(t)),\nonumber\\
M^*(N) &= \tilde{S},
\label{eq:optgainmatrix}
\end{align}
where,
$\tilde{A} = \begin{bmatrix}
A & 0 \\
0 & A
\end{bmatrix}$, $G(t) = \begin{bmatrix}
BK(t)&0\\
0&B\begin{bmatrix}
0\\
\hat{K}_{21}(t)&J(t)
\end{bmatrix}
\end{bmatrix}$, $\tilde{B} = \begin{bmatrix}
0\\
B\end{bmatrix}$,

$H(t) = \begin{bmatrix}
K & \begin{bmatrix}
0\\
\hat{K}_{21}(t) & J(t)
\end{bmatrix}
\end{bmatrix}$, $\tilde{Q} = \begin{bmatrix}
Q & 0\\
0 & 0
\end{bmatrix}$, $\tilde{S} = \begin{bmatrix}
S & 0\\
0 & 0
\end{bmatrix}$,

$V(t) = \mathbf{E}[{\phi_o}(t)n_1(t)']$, and $n_1(t) := \begin{bmatrix}
\sR(t)\phi(t)\\
\overline{\sR}(t)\overline{\phi}(t) - \sR(t)\phi(t)
\end{bmatrix}$.

\begin{remark}
As in Theorem 2, it is not difficult to obtain complete linear expressions for $F^*(t)$ from it using the sparsity constraints on $\Psi^*(t)$ and $F^*(t)$.
\end{remark}
\end{theorem}
\vspace*{-0.25in}
\section{The $(1,0)$ information sharing pattern}\label{sec:problem10}
\vspace*{-0.05in}
Another application of the above approach, which is also an extension to~\cite{KuSi74}, is the $(1,0)$-pattern. A two-player discrete linear time-invariant system with output feedback is considered. The system dynamics are,
\begin{align}\label{eq:system10}
x(t+1) &= A x(t) + B u(t) + v(t), \nonumber \\
y(t) &= Cx(t) + w(t)
\end{align}

The two players have a (team/common) objective to find the decentralized control law $(u^*_1(\cdot),u^*_2(\cdot))$ that minimizes a finite-time quadratic cost criterion
\begin{equation}\label{eq:costfn10}
\bbE[\sum_{t=0}^{N-1}(x(t)'Qx(t) + u(t)'Ru(t)) + x(N)'Sx(N)],
\end{equation}

We consider the following information structure:

\begin{align}\label{eq:infostructure10}
\sF_1(t) &= \{y_1(0:t),u_1(0:t-1),y_2(0:t),u_2(0:t-1)\}, \nonumber \\
\sF_2(t) &= \{y_1(0:t-1),u_1(0:t-1),\nonumber\\
&~~~~~y_2(0:t),u_2(0:t-1)\}.
\end{align}
where $y(0:t)$ denotes the vector $(y(0),\cdots,y(t))$.

We refer to this as the $(1,0)$ information sharing pattern. Finding the optimal decentralized controller under this information pattern is similar to that of the $(1,1)$ pattern~\cite{KuSi74}, with the difference being the asymmetry in observation history. We outline how the same approach that we used for the $(1, \infty)$ scenario can be applied also to this case.


\subsubsection{Linearity of the optimal control law}\label{sec:linearity}



\begin{proposition}\label{lem:linearity10}
The LQG problem~\eqref{eq:system10}-\eqref{eq:costfn10} with the $(1,0)$-delayed sharing information pattern~\eqref{eq:infostructure10} has a partially nested structure, and a linear optimal control law exists.
\end{proposition}
See, for instance, \cite[Theorem 2]{HoCh72} for a proof.
%

Denote by $\sH(t) = ( y(0:t-1), u(0:t-1))$, the common information at time $t$~\cite{Wi71}. Using the linearity of the optimal control law, it follows that the optimal control law can be written as $u(t)=[u_1(t),~u_2(t)]'=$
\begin{equation}\label{eq:optlinlaw10}
\begin{bmatrix}
F^*_{11}(t) & F^*_{12}(t)\\
0 & F^*_{22}(t)
\end{bmatrix}
\begin{bmatrix}
y_1(t)\\
y_2(t)
\end{bmatrix}
+
\begin{bmatrix}
L_1(\sH(t)) \\
L_2(\sH(t))
\end{bmatrix},
\end{equation}
where $F^*$ is the optimal gain matrix and $L_i(\cdot)$ are linear, possibly time-varying, functions.

\subsubsection{Derivation of the Optimal Control Law}\label{sec:optimality}
Define the estimator $\hat{x}(t) := \bbE[x(t)|\sH(t)]$.

\begin{lemma}\label{lem:estimatordynamics10}
\cite[Section 6.5]{KwSi72} For the nested system~\eqref{eq:system10}, the estimator dynamics for $\hat{x}(t)$ are as follows.
\begin{equation}
\hat{x}(t+1) = A \hat{x}(t) +  B u(t) + \sK(t)\phi(t),
\label{eq:estdynO10}
\end{equation}
and $\hat{x}(0) = \bbE[x(0)]$, where, $\sK(t) = K_{xy}(t)K_{y}(t)^{-1}$ and $\phi(t) = y(t) - C\hat{x}(t)$. $K_{xy}(t) = AT(t)C'$, $K_{y}(t) = CT(t)C' + W$. $T(t) = \bbE[(x(t) - \hat{x}(t))(x(t) - \hat{x}(t))]$. $\phi(t)$ is zero-mean, uncorrelated Gaussian and $\bbE[\phi(t)\phi(t)'] = CT(t)C' + W$. $\bbE[\phi(t)y(\t)'] = 0$, $\tau < t$.
\end{lemma}

\begin{theorem}
For the system~\eqref{eq:system10} with the information structure~\eqref{eq:infostructure10} and objective function~\eqref{eq:costfn10},
\begin{equation}
\begin{bmatrix}
u_1^*(t)\\
u_2^*(t)
\end{bmatrix} = F^*(t) \begin{bmatrix}
y_1(t)\\
y_2(t)
\end{bmatrix} + G^*(t)\hat{x}(t),
\label{eq:optcontroller10}
\end{equation}
where $F^*(t) = \begin{bmatrix}
F^*_{11}(t) &  F_{12}^*(t)\\
\mathbf{0} & F^*_{22}(t)
\end{bmatrix}$ is the optimal gain and $G^*(t) = - (F^*(t)C + H^*(t))$ where $H^*(t)$ is the gain for classical LQR with dynamics described by~\eqref{eq:system10}.
\label{thm:optcontrollerformO10}
\end{theorem}
\vspace*{-0.1in}
\label{sec:optgain}
\begin{proposition}
The optimal gain matrix $F^*(t)$ is given by the recursion,

\begin{align}
&F^*(t) = -\Big(R + B'M(t+1)B\Big)^{-1}\times\nonumber\\
&~~~~~\Big(\Psi^*(t)(CT(t)C' + W(t))^{-1} + B'M(t+1)\sR(t)\Big),\nonumber\\
&M^*(t) = Q + H(t)'RH(t) \nonumber\\
&~~~~+ (A - BH(t))'M^*(t+1)(A - BH(t)),\nonumber\\
&M^*(N) = \tilde{S}.
\label{eq:optgainmatrix10}
\end{align}
\end{proposition}
\vspace*{-0.1in}

From Theorems 4 and 5, we see that the form of the optimal controller for the $(1,\infty)$ case is a mixture of those for the $(1,0)$ and the $(0,\infty)$ case. Specifically, the first component that includes the innovation sequence resembles the $(1,0)$ setting, and the second component resembles the $(0,\infty)$ solution. In fact, it can be seen from Theorem 4 that removing the innovation sequence term (as is the case when there are no delays) results in the same solution as the $(0,\infty)$ case.

\vspace*{-0.1in}
\section{Future Work}\label{sec:conclusions}
\vspace*{-0.05in}


We believe that the approach of combining linearity and summary statistics in the manner of this paper can be used to compute optimal control laws for more general asymmetric information sharing patterns. Specifically, the $n$-player network extension to the output feedback scenario would be a novel extension to~\cite{LaLe14}.
\vspace*{-0.1in}

\section{Acknowledgments}\label{sec:acknowledgments}
We gratefully acknowledge the contributions of the editor and our reviewers to this paper, in particular to the reviewer who pointed out an approach to use sparsity patterns to solve the gain matrix equations.
\vspace*{-0.14in}
\appendices

\section{Proofs of Theorems and Lemmas}
\vspace*{-0.05in}
\begin{proof}[Proof of lemma~\ref{lem:estimatordynamics}]
Let us first derive the dynamics of block $1$. Denote $\phi(t) := x_1(t) - \hat{x}_1(t)$. Then, $\phi(t)$ and $\sH_1(t)$ are independent by the projection theorem for Gaussian random variables \cite[Section 6.5]{KwSi72} and,
\begin{align*}
\hat{x}(t+1) &= \bbE[x(t+1) | \sH_1(t+1)]\\
&= \bbE[x(t+1) | \sH_1(t), x_1(t), u_1(t)]\\
&= \bbE[x(t+1) | \sH_1(t),u_1(t)] + \bbE[x(t+1) | \phi(t)]\nonumber\\
&~~ - \bbE[x(t+1)],
\end{align*}
\vspace*{-0.05in}
where the last equality follows from the independence of $\phi(t)$ and $\sH_1(t)$~\cite[Section 6.5]{KwSi72}.
The first term on the RHS above, $\bbE[x(t+1) | \sH_1(t),u_1(t)]$ 
\begin{align*}
&= \bbE[Ax(t) + Bu(t) | \sH_1(t),u_1(t)]\\
&= A\hat{x}(t) + B\begin{bmatrix} u_1(t)\\
\hat{u}_2(t)
\end{bmatrix}, 
\end{align*}
where $\hat{u}_2(t) = \bbE[u_2(t) | \sH_1(t)]$.

The second and third terms can be related through the conditional estimation of Gaussian random vectors as,
\begin{align*}
\bbE[x(t+1) | \phi(t)] - \bbE[x(t+1)] = R_{xx_1}R_{x_1}^{-1}\phi(t),
\end{align*}
where defining  $T(t) = \bbE[(x(t) - \hat{x}(t))(x_1(t) - \hat{x}_1(t))]$ and partitioning $T(t)$ as $T(t) = \big[T_1(t) | T_2(t)\big]$, we have $R_{xx_1} = AT(t)$, $R_{x_1} = T_1(t)$ ~\cite[Section 6.5]{KwSi72}.

Putting these three terms together, we have
\begin{equation*}
\hat{x}(t+1) = A \hat{x}(t) +  B\begin{bmatrix} u_1(t)\\
\hat{u}_2(t)
\end{bmatrix} + \sR(t)\phi(t).
\end{equation*}
Similarly, by defining $\overline{\phi}(t) = x(t) - \doublehat{x}(t)$ and proceeding as above, we have
\begin{equation*}
\doublehat{x}(t+1) = A \doublehat{x}(t) +  B\begin{bmatrix} u_1(t)\\
u_2(t)
\end{bmatrix} + \overline{\sR}(t) \overline{\phi}(t).
\end{equation*}

To derive the properties of $\phi(t)$ and $\overline{\phi}(t)$, note that the mean and variance follow immediately by observing that $\bbE[\hat{x}_1(t)] = \bbE[x_1(t)]$ and $\bbE[\doublehat{x}(t)]  = \bbE[x(t)]$.

The projection theorem for Gaussian RVs again implies that $\phi(t)$ and $x_1(\t)$ are uncorrelated for $\tau < t$, and that $\overline{\phi}(t)$ is uncorrelated with $x(\t)$, $\tau < t$.
\end{proof}


\vspace*{-0.15in}
\begin{proof}[Proof of Theorem~\ref{thm:optcontrollerform}]
We begin by outlining our proof technique. The approach, while superficially similar to~\cite{KuSi74} is more challenging due to the asymmetry in the observation history. The optimal control law for the $(1,\infty)$ state-feedback pattern, already shown to be linear since the problem is partially nested, is proven to be the same as the optimal control law for the dynamical system~\eqref{eq:nestedsystem} with the objective function $\arg\min_{\overline{u}(t)}\bbE[\sum_{t=0}^{N-1}(\overline{x}(t)'Q\overline{x}(t) + \overline{u}(t)'R\overline{u}(t)) + \overline{x}(N)'S\overline{x}(N)]$, where $\overline{x}(t) = [\hat{x}_1(t)~\doublehat{x}_2(t)]$ and $\overline{u}(t)$ is an input variable that is an invertible transformation of the original input $u(t)$. The crucial point behind this transformation of the objective function and the input variable is that $\overline{u}(t)$ will be shown to be a function of just the summary statistics $\hat{x}(t)$ and $\doublehat{x}(t)$, hence allowing efficient computation of the control law.

Since the new objective depends on $\overline{x}(t)$ and $\overline{u}(t)$, the dynamics of $x(t)$ can be rewritten in terms of these. It can then be shown that the modified problem is of the form as in Proposition~\ref{lem:lall}. Thus, the optimal control law can be obtained for the modified problem, and inverted to get the optimal law for the $(1,\infty)$ state-feedback pattern.

As a preliminary result, the following lemma derives a required uncorrelatedness property and also proves the equivalence of two of the estimator variables.
\vspace*{-0.05in}
\begin{lemma}\label{lem:equivestimators}
For the nested system~\eqref{eq:nestedsystem}, $\hat{x}_1(t) = \doublehat{x}_1(t)$. Additionally, $\phi(t) = \overline{\phi}_1(t)$, which implies from Lemma~\ref{lem:estimatordynamics} that $\phi(t)$ is uncorrelated with $\sH_2(t)$, and so $\bbE[\phi(t)x(\tau)'] = \mathbf{0}, \tau < t$.
\end{lemma}
\begin{proof}
Intuitively, the claim is true because Player $2$ has no additional information about block $1$'s dynamics as compared to Player $1$. Formally,
\begin{align*}
\doublehat{x}_1(t) 
&= \bbE[x_1(t) | \sH_1(t), \sH_2(t)],\\
&= A_{11}x_1(t-1) + B_{11}u_1(t-1)= \hat{x}_1(t), 
\end{align*}
where we have used the independence of the zero-mean noise $v_1(t)$ from both $\sH_1(t)$ and $\sH_2(t)$, and the definitions of  $\phi(t)$ and $\overline{\phi}(t)$.
\end{proof}

Note that the estimator dynamics of $[\hat{x}_1(t),~\doublehat{x}_2(t)]$ can be written in the following way,
\begin{align}
\begin{bmatrix}
\hat{x}_1(t+1)\\
\doublehat{x}_2(t+1)
\end{bmatrix} =& \begin{bmatrix}
A_{11} & 0\\
A_{21} & A_{22}
\end{bmatrix} \begin{bmatrix}
\hat{x}_1(t)\\
\doublehat{x}_2(t)
\end{bmatrix}\nonumber\\
 +& \begin{bmatrix}
B_{11} & 0\\
B_{21} & B_{22}
\end{bmatrix} \begin{bmatrix}
u_1(t)\\
u_2(t)
\end{bmatrix} \nonumber\\
+& \begin{bmatrix}
(\sR(t)\phi(t))_1\\
(\overline{\sR}(t)\overline{\phi}(t))_2
\end{bmatrix}.
\label{eq:estimatordynamics}
\end{align}

This shall be denoted in shorthand notation as,
\begin{equation*}
\overline{x}(t+1) = A \overline{x}(t) + Bu(t) +
 \begin{bmatrix}
(\sR(t)\phi(t))_1\\
(\overline{\sR}(t)\overline{\phi}(t))_2
\end{bmatrix}.
\end{equation*}

Let us denote estimation error by $e(t) = [e_1(t)~e_2(t)] := [x_1(t) - \hat{x}_1(t)~x_2(t) - \doublehat{x}_2(t)]$ respectively. Then, $e_1(t) = \phi_1(t)$ and  $e_2(t) = \overline{\phi}_2(t)$. By the projection theorem and some manipulations, it can be verified that $\bbE[e_i(t)\overline{x}_j(t)'] = 0$ for $i,j = 1,2$.

Further, we define a transformed system input $\overline{u}(t)$ by
\begin{equation}
\overline{u}(t) = \begin{bmatrix}
\overline{u}_1(t)\\
\overline{u}_2(t)
\end{bmatrix} := u(t) -
F^*(t)e(t),
\label{eq:modifiedinput}
\end{equation}
where $F^*(t)$ is the optimal gain matrix in~\eqref{eq:optlinlaw}. Note that the transformation is invertible. From~\eqref{eq:optlinlaw}, it can  be observed that two components of $\overline{u}(t)$ are linear functions of $\sH_1(t)$ and $(\sH_1(t), \sH_2(t))$ respectively, since,
\begin{align}
\overline{u}(t) &= u(t) - F^*(t)\begin{bmatrix}
x_1(t) - \hat{x}_1(t)\nonumber\\
x_2(t) - \doublehat{x}_2(t)
\end{bmatrix}\\
&= \begin{bmatrix}
F^*_{11}(t)\hat{x}_1(t) + L_1(\sH_1(t)) \\
F^*_{22}(t)\doublehat{x}_2(t) + L_2(\sH_1(t),\sH_2(t))
\end{bmatrix}\nonumber\\
&= \begin{bmatrix}
L'_1(\sH_1(t)) \\
L'_2(\sH_1(t),\sH_2(t))
\end{bmatrix},
\end{align}
since $\hat{x}(t)$ and $\doublehat{x}(t)$ are linear functions of $\sH_1(t)$ and $(\sH_1(t), \sH_2(t))$, respectively.

Using transformed input $\overline{u}(t)$,~\eqref{eq:estimatordynamics} can be rewritten as
\begin{equation}\label{eq:estimatordynamicsA}
\overline{x}(t+1) = A \overline{x}(t) + B \overline{u}(t) + \tilde{v}(t),
\end{equation}
where $\tilde{v}(t)$ is equal to
\begin{equation*}
\begin{bmatrix}
B_{11}F_{11}(t)\phi(t) + (\sR(t)\phi(t))_1\\
B_{21}F_{11}(t)\phi(t) + B_{22}F_{22}(t)\overline{\phi}_2(t) + (\overline{\sR}(t)\overline{\phi}(t))_2
\end{bmatrix}.
\end{equation*}

It can be verified from Lemmas~\ref{lem:estimatordynamics} and~\ref{lem:equivestimators} that $\tilde{v}(t)$ is Gaussian, zero-mean and independent across time. With the estimator system well-characterized, we now show that the original objective function~\eqref{eq:costfn} is equivalently expressed in terms of $\overline{x}(t)$ and $\overline{u}(t)$.

Rewriting~\eqref{eq:costfn} in terms of the new state variable $\overline{x}(t)$ and the transformed input $\overline{u}(t)$, after noting that the terms involving $e_1(t)$ and $e_2(t)$ are either zero or independent of the input $\overline{u}(t)$, we have,
\vspace*{-0.1in}
\begin{align}
\label{eq:costfnAo}
&\min_{u(t), 0\leq t\leq N-1} \bbE[\sum_{t=0}^{N-1}(x(t)'Qx(t) + u(t)'Ru(t))\nonumber\\
&\hspace{0.5in} + x(N)'Sx(N)]\nonumber\\
= &\min_{\overline{u}(t), 0\leq t\leq N-1} \bbE[\sum_{t=0}^{N-1}(\overline{x}(t)'Q\overline{x}(t) + \overline{u}(t)'R\overline{u}(t)'\nonumber\\
&\hspace{0.5in} +e(t)' F^*(t)RF^*(t)e(t) \\
&\hspace{0.5in} + 2 e(t)'F^*(t)R\overline{u}(t)) +\overline{x}(N)'S\overline{x}(N)].\nonumber
\end{align}

\vspace*{-0.1in}
The term $e(t)' F^*(t)RF^*(t)e(t)$ is an estimation error and is independent of the control input $\overline{u}(t)$. This can be seen from the fact that $e_1(t) = \phi_1(t)$ and $e_2(t) = \overline{\phi}_2(t)$ are independent of $\sH_1(t)$ and $\sH_2(t)$, by Lemmas~\ref{lem:estimatordynamics} and~\ref{lem:equivestimators}. $\overline{u}(t)$ is a linear function of $\sH_1(t)$ and $\sH_2(t)$ and so is independent of $e(t)$. In a sense, we have ``subtracted'' out the dependence on $e(t)$ in~\eqref{eq:modifiedinput}.

$\overline{u}_1(t)$ and $\overline{u}_2(t)$ are linear functions of $\sH_1(t)$ and $(\sH_1(t),\sH_2(t))$ respectively. So the fourth term has zero expected value since $\bbE[\phi(t)x(\tau)'] = \mathbf{0} = \bbE[\overline{\phi}(t)x(\tau)'], \tau < t$, from Lemmas~\ref{lem:estimatordynamics} and~\ref{lem:equivestimators}. Thus, the following is an equivalent function,
\begin{align}
\min_{\overline{u}(t)} \bbE[\sum_{t=0}^{N-1}(\overline{x}(t)'Q\overline{x}(t) + \overline{u}(t)'R\overline{u}(t)) + \overline{x}(N)'S\overline{x}(N)].
\label{eq:costfnA}
\end{align}

\vspace*{-0.1in}
Observe that, with this transformation, solving the $(1,\infty)$ problem is equivalent to solving the cost ~\eqref{eq:costfnA} for the system~\eqref{eq:estimatordynamicsA} with no communication delay ($(0,\infty)$-delayed sharing pattern) in the new state variables. A similar observation has been remarked in~\cite[IV]{LaDo13}.

Letting $[\overline{u}_1(t)~\overline{u}_2(t)]'$ take up the role of players' inputs,~\eqref{eq:estimatordynamicsA} has a nested system structure with no communication delays. By Proposition~\ref{lem:lall}, the optimal controller for the system~\eqref{eq:estimatordynamicsA}-\eqref{eq:costfnA} is given by,
\begin{equation}\label{eq:optmodified}
\begin{bmatrix}
\overline{u}_1^*(t)\\
\overline{u}_2^*(t)
\end{bmatrix} = \begin{bmatrix}
K_{11}(t) & K_{12}(t) & \mathbf{0} \\
K_{21}(t) & K_{22}(t) & J(t)
\end{bmatrix} \begin{bmatrix}
\hat{x}_1(t)\\
\hat{x}_2(t)\\
\hat{\overline{x}}_2(t) - \hat{x}_2(t)
\end{bmatrix},
\end{equation}
where $\hat{\overline{x}}_2(t) = \bbE[\doublehat{x}_2(t) | \overline{x}_1(0:t), \overline{u}_1(0:t-1)]$. As shown below, $\hat{\overline{x}}_2(t)$ is just $\hat{x}_2(t)$.
\vspace*{-0.05in}
\begin{align*}
\hat{x}_2(t) &= \bbE[x_2(t) | \sH_1(t)],\\
&= \bbE\big[\bbE[x_2(t) | \sH_1(t) \big] \big| \overline{x}_1(0:t), \overline{u}_1(0:t-1)],\\
&= \bbE[x_2(t) | \overline{x}_1(0:t), \overline{u}_1(0:t-1)]],\\
&= \bbE\big[\bbE[x_2(t) | \sH_1(t),\sH_2(t)\big] \big| \overline{x}_1(0:t), \overline{u}_1(0:t-1)],\\
&= \hat{\overline{x}}_2(t),
\end{align*}
\vspace*{-0.025in}
where, in the second equality, the conditional expectation is taken w.r.t. itself and other random variables. The $3^{rd}$ and $4^{th}$ equalities follow from the tower rule.

Recall that $\overline{u}(t)$ was allowed to be any linear function of the histories $\sH_1(t)$ and $\sH_2(t)$. However, $\overline{u}(t)$ finally depends on just the summary statistics $\hat{x}(t)$ and $\doublehat{x}(t)$, thereby proving that these information statistics are indeed sufficient to compute the optimal law.

From~\eqref{eq:optmodified} and~\eqref{eq:modifiedinput}, $u(t)$ is obtained as,
\begin{equation}\label{eq:optcontroller}
u(t)  =  F^*(t)\begin{bmatrix}
x_1(t) - \hat{x}_1(t)\\
x_2(t) - \doublehat{x}_2(t)
\end{bmatrix} + \tilde{K}(t).
\end{equation}
\end{proof}
\vspace*{-0.35in}
\begin{proof}[Proof of Theorem~\ref{thm:optgainmatrix}]
Assuming that optimal gain matrix $F^*(t)$ is unknown, we replace it by arbitrary $F(t)$ and solve for it by deterministic matrix optimization.

From Theorem~\ref{thm:optcontrollerform}, using the modified input $\overline{u}(t) := u(t) - \begin{bmatrix}
F_{11}(t) \phi(t)\\
F_{22}(t) \overline{\phi}_2(t)
\end{bmatrix}$, we have,
\begin{equation}
\label{eq:shortformlall}
\overline{u}(t) = H(t) \tilde{x}(t),
\end{equation}
where $H(t) = \begin{bmatrix}
K_{11}(t) & K_{12}(t) & -K_{12}(t) \\
K_{21}(t) & K_{22}(t) & J(t) - K_{22}(t)
\end{bmatrix}$, and $\tilde{x}(t) := \begin{bmatrix}
\hat{x}_1(t)\\
\doublehat{x}_2(t)\\
\doublehat{x}_2(t) - \hat{x}_2(t)
\end{bmatrix}$.

The dynamics of $\tilde{x}(t)$ are obtained from~\eqref{eq:estdyn2},~\eqref{eq:estimatordynamicsA} as,
\begin{align}
\label{eq:dynamicsFopt}
\tilde{x}(t+1) = \big(\tilde{A} + G(t)\big)\tilde{x}(t) + n(t),
\end{align}
where $\tilde{A} := \begin{bmatrix}
A & \mathbf{0}\\
\mathbf{0} & A_{22}
\end{bmatrix}$, $G(t) := \begin{bmatrix}
&BH(t)&\\
0&0&B_{22}J(t)
\end{bmatrix}$, and $n(t) := \begin{bmatrix}
\tilde{v}(t)\\
\tilde{v}_2(t) - (\sR(t)\phi(t))_2
\end{bmatrix}$.

Let $\sigma_{F(t)}$ denote the variance of $n(t)$.
Since $F(t)$ minimizes~\eqref{eq:costfnAo}, keeping only terms that depend on $F(t)$,
\begin{align}\label{eq:costfnFoptA}
J_F = \min_{F(t)}\bbE\bigg[\sum_{t=0}^{N-1}\Big(\tilde{x}(t)'\tilde{Q}\tilde{x}(t) + \tilde{x}(t)' H(t)' R H(t) \tilde{x}(t) \nonumber \\
+\overline{\phi}(t)' F(t)'RF(t)\overline{\phi}(t)\Big) + \tilde{x}(N)'\tilde{S}\tilde{x}(N)\bigg],\nonumber\\
\end{align}
where $\tilde{Q} = \begin{bmatrix}
Q & 0\\
0 & 0
\end{bmatrix}$ and $\tilde{S} = \begin{bmatrix}
S & 0\\
0 & 0
\end{bmatrix}$.
Using matrix algebra,~\ref{eq:costfnFoptA} simplifies to,
\begin{align}
J_F= \min_{F(t)}\big[\sum_{t=0}^{N-1}\text{trace}\Big(\big(\tilde{Q} + H(t)' R H(t)\big)\Sigma(t)\Big)\nonumber\\
+\text{trace}\Big(\tilde{S}\Sigma(N)\Big)\nonumber + \sum_{t=0}^{N-1}\text{trace}\Big(F(t)'RF(t)\overline{T}(t)\Big)\big],
\end{align}
where $\Sigma(t) := \bbE[\overline{x}(t)\overline{x}(t)']$, with dynamics from~\eqref{eq:dynamicsFopt},

\vspace*{-0.25in}
\begin{align}
\Sigma(t+1) &= \big(\tilde{A} + G(t)\big)\Sigma(t)\big(\tilde{A} + G(t)\big)' + \sigma_{F(t)},\nonumber\\
\Sigma(0) &= \bbE[\overline{x}(0)\overline{x}(0)'] = \mathbf{0}.
\label{eq:dynamicsFoptA}
\end{align}
where $\sigma_{F(t)}$ is the covariance of $n(t)$. $F^*(t)$ may now be obtained using the discrete matrix minimum principle~\cite{KlAt66} through the Hamiltonian to minimize~\eqref{eq:costfnFoptA}.
\begin{align}
\sH &= \text{trace}[(\tilde{Q} + H(t)'RH(t))\Sigma (t)] \nonumber\\
&+ \text{trace}[F(t)'RF(t)\overline{T}(t)] \nonumber\\
&+\text{trace}[(\tilde{A} + G(t))\Sigma (t)(\tilde{A} + G(t))'M(t+1)'] \nonumber\\
&+\text{trace}[\sigma_{F(t)}M(t+1)'] + \text{trace}[2F(t)\Psi(t)']
\end{align}
where $\Psi(t) = \begin{bmatrix}
0 & \Psi_{12}(t)\\
\Psi_{21}(t) & 0
\end{bmatrix}$ is a suitably partitioned Lagrange multiplier matrix and $M(t)$ is the costate matrix. The optimality criteria are,
\begin{align*}
\frac{\partial \sH}{\partial F(t)} = 0, &~\frac{\partial \sH}{\partial \Psi(t)} = 0,~
\frac{\partial \sH}{\partial M(t)} = \Sigma^*(t),~\Sigma^*(0) = \Sigma(0)\\
\frac{\partial \sH}{\partial \Sigma(t)} &= M^*(t),~M^*(N) = \frac{\partial \text{trace}[S\Sigma(N)]}{\partial \Sigma(N)}
\end{align*}

Solving the above equations, we get,
\begin{align}
&2RF^*(t)\overline{T}(t) + \frac{\partial \text{trace}[\sigma_{F(t)}M(t+1)']}{\partial F(t)} + 2 \Psi^*(t) = 0\nonumber\\
&F^*_{12}(t) = F^*_{21}(t) = 0\nonumber\\
&\Sigma^*(t+1) = (\tilde{A} + G(t))\Sigma (t)(\tilde{A} + G(t))' + \sigma_{F(t)},\nonumber\\
&\Sigma^*(0) = 0\nonumber\\
&M^*(t) = \tilde{Q} + H(t)'RH(t) \nonumber\\
&\hspace{0.1in}+ (\tilde{A} + G(t))'M^*(t+1)(\tilde{A} + G(t)), M^*(N) = S.
\label{eq:optcondns}
\end{align}
The noise term in ~\eqref{eq:dynamicsFopt} can be rewritten as, $n(t)$=
 
$\begin{bmatrix}
B_{11}F_{11}(t)\overline{\phi}_1(t) + (\sR(t)\overline{\phi}_1(t))_1\\
B_{21}F_{11}(t)\overline{\phi}_1(t) + B_{22}F_{22}(t)\overline{\phi}_2(t) + (\overline{\sR}(t)\overline{\phi}(t))_2\\
B_{22}F_{22}(t)\overline{\phi}_2(t) + (\overline{\sR}(t)\overline{\phi}(t))_2 - (\sR(t)\overline{\phi}_1(t))_2\end{bmatrix} = \tilde{B}F(t)\overline{\phi}(t) + n_1(t)$, \\
where $\tilde{B} := \begin{bmatrix}
B_{11} & 0\\
B_{21} & B_{22}\\
0 & B_{22}
\end{bmatrix}$, and $n_1(t) := \begin{bmatrix}
(\sR(t)\overline{\phi}_1(t))_1\\
(\overline{\sR}(t)\overline{\phi}(t))_2\\
(\overline{\sR}(t)\overline{\phi}(t))_2 - (\sR(t)\overline{\phi}_1(t))_2
\end{bmatrix}$.
Then,
\begin{align*}
\sigma_{F(t)} =& \bbE[(\tilde{B}F\overline{\phi}(t) + n_1(t))(\tilde{B}F\overline{\phi}(t) + n_1(t))']\\
\equiv& \bbE[\tilde{B}F(t)\overline{\phi}(t)n_1(t)' + \tilde{B}F(t)\overline{\phi}(t)(\tilde{B}F(t)\overline{\phi}(t))' \\
&~~+ n_1(t)n_1(t)' + n_1(t)(\tilde{B}F(t)\overline{\phi}(t))']\\
=& \tilde{B}F(t)\bbE[\overline{\phi}(t)n_1(t)'] + \tilde{B}F(t)\bbE[\overline{\phi}(t)\overline{\phi}(t)']F(t)'\tilde{B}' \\
&~~+ \bbE[n_1(t)\overline{\phi}(t)']F(t)'\tilde{B}'\\
=& \tilde{B}F(t)V(t) + \tilde{B}F(t)\overline{T}(t)F(t)'\tilde{B}' + V(t)'F(t)'\tilde{B}'
\end{align*}
where $V(t) = \bbE[\overline{\phi}(t)n_1(t)']$ and equivalence in the $2^{nd}$ line is from taking partial derivative of $\sigma_{F(t)}$ w.r.t. $F(t)$.

Taking the partial derivative of $\text{trace}[\sigma_{F(t)}M(t+1)']$,
\begin{align*}
&\frac{\partial \text{trace}[\sigma_{F(t)}M(t+1)']}{\partial F(t)} \\
=& \frac{\partial}{\partial F(t)}\Big(\text{trace}[(\tilde{B}F(t)V(t) + \tilde{B}F(t)\overline{T}(t)F(t)'\tilde{B}' \\
&\hspace{0.5in}+ V(t)'F(t)'\tilde{B}')M(t+1)']\Big)\\
=& \tilde{B}'M(t+1)V(t)' + \tilde{B}'M(t+1)\tilde{B}F(t)\overline{T}(t) \\
&+ \tilde{B}'M(t+1)'\tilde{B}F(t)\overline{T}(t) + \tilde{B}'M(t+1)'V(t)'
\end{align*}

Substituting into the optimality condition \eqref{eq:optcondns} we get,
\begin{align}
F^*(t) =& - \Big(2R + \tilde{B}'M(t+1)\tilde{B} + \tilde{B}'M(t+1)'\tilde{B}\Big)^{-1}\times\nonumber\\
&~~~~\Big(2 \Psi^*(t) + \tilde{B}'M(t+1)V(t)' \nonumber\\
&\hspace{0.5in}+ \tilde{B}'M(t+1)'V(t)'\Big)\overline{T}^{-1}(t)\nonumber\\
=& -\Big(R + \tilde{B}'M(t+1)\tilde{B}\Big)^{-1}\times\nonumber\\
&~~~~\Big(\Psi^*(t) + \tilde{B}'M(t+1)V(t)'\Big)\overline{T}^{-1}(t)
\label{eq:optgainmatrix}
\end{align}
where the last equality uses $M(t)$ being symmetric.

The following is a method to obtain the gain matrix $F^*(t)$ explicitly from the above equation. Rearranging~\eqref{eq:optgainmatrix} results in the following equation:
\begin{align}
\Big(R + \tilde{B}'M(t+1)\tilde{B}\Big)F^*(t)\overline{T}(t) + \Psi^*(t) \nonumber\\
 + \tilde{B}'M(t+1)V(t)'\Big)
\end{align}

Using $\Psi_{11}^*(t) = \Psi_{22}^*(t) = 0$, we get,
\begin{align}
E_1' \Big(R + \tilde{B}'M(t+1)\tilde{B}\Big)F^*(t)\overline{T}(t) E_1 \nonumber\\
 = -E_1' \tilde{B}'M(t+1)V(t)'\Big) E_1 \\
E_2' \Big(R + \tilde{B}'M(t+1)\tilde{B}\Big)F^*(t)\overline{T}(t) E_2 \nonumber\\
 = -E_2' \tilde{B}'M(t+1)V(t)'\Big) E_2,
\end{align}
where $E_1$ and $E_2$ are suitable partitions of the identity matrix $I = [ E_1 ~~ E_2]$. We also have, $F_{12}^*(t) = F_{21}^*(t) = 0$, giving us $F^*(t)  = E_1F_{11}^*(t)E_1' + E_2 F_{22}^*(t)E_2'$. Substituting into the above equation, we get,
\begin{align}
\Big[R + \tilde{B}'M(t+1)\tilde{B}\Big]_{11}F_{11}^*(t)\overline{T}_{11}(t) \nonumber \\
+ \Big[R + \tilde{B}'M(t+1)\tilde{B}\Big]_{12}F_{22}^*(t)\overline{T}_{21}(t) \nonumber\\
 = -\Big[\tilde{B}'M(t+1)V(t)'\Big]_{11} \\
\Big[R + \tilde{B}'M(t+1)\tilde{B}\Big]_{21}F_{11}^*(t)\overline{T}_{12}(t) \nonumber \\
+ \Big[R + \tilde{B}'M(t+1)\tilde{B}\Big]_{22}F_{22}^*(t)\overline{T}_{22}(t) \nonumber\\
 = -\Big[\tilde{B}'M(t+1)V(t)'\Big]_{22},
\end{align}
where $[A]_{ij} = E_i'AE_j$. This is a complete set of linear equations that can be solved directly, for example, by vectorization using the Kronecker product $A X B = C \Leftrightarrow (B'\otimes A)x = c$, where $\otimes$ is the Kronecker product and $x$, $c$ are vectors made by assembling columns of $X$, $C$.
\end{proof}


\begin{proof}[Proof of Theorem~\ref{thm:optcontrollerformO}]
Only the outline of the proof is given here for brevity. As in the proof of Theorem~\ref{thm:optcontrollerform}, we define a modified input,
\begin{align*}
\overline{u}(t) := u(t) - F^*(t)\phi_o(t),
\end{align*}
where $\phi_o(t) := \begin{bmatrix}
y_1(t) - C_{11}\hat{x}_1(t)\\
y_2(t) - C_{21}\doublehat{x}_1(t) - C_{22}\doublehat{x}_2(t)
\end{bmatrix}$.

It can be shown that $\overline{u}(t)$ can be decomposed into a coordinator's action, based on common history $\sH_1(t)$, and the remaining part based on the entire information available to player $2$~\cite[Theorem 10]{LeNa13}. That is,
\begin{equation*}
\overline{u}(t) = \tilde{u}(t) + \tilde{\tilde{u}}(t),
\end{equation*}
where $\tilde{u}(t)$ is obtained by fixing $\overline{u}(t) := \tilde{u}(t) + \hat{K}(t)\doublehat{x}(t)$ and solving the optimal control problem. Subsequently, $\tilde{\tilde{u}}(t)$ is obtained by fixing $\overline{u}(t) := K(t)\hat{x}(t) + [0~~\tilde{\tilde{u}}]'$ and solving the LQR problem~\cite[Theorem 10]{LeNa13}. Person-by-person optimality implies global optimality since the problem is partially nested~\cite{MaMaRo12}.

It can be verified through the above approach that the additional $F^*(t)\phi_o(t)$ term does not affect the computation of $\tilde{u}(t)$ and $\tilde{\tilde{u}}(t)$ and,
\begin{align*}
\tilde{u}(t) &= (K(t) - \hat{K}(t))\hat{x}(t), \tilde{\tilde{u}}(t) &= \begin{bmatrix}
0\\
\hat{K}(t)(\doublehat{x}(t) - \hat{x}(t))
\end{bmatrix}
\end{align*}
\end{proof}

\bibliographystyle{ieeetr}
\bibliography{main-tcns14-decentralized-R4}

\begin{thebibliography}{10}

\bibitem{RoLa06}
M.~Rotkowitz and S.~Lall, ``A characterization of convex problems in
  decentralized control,'' {\em IEEE Transactions on Automatic Control},
  vol.~51, no.~2, pp.~274--286, 2006.

\bibitem{LaChDa04}
C.~Langbort, R.~S. Chandra, and R.~D'Andrea, ``Distributed control design for
  systems interconnected over an arbitrary graph,'' {\em IEEE Transactions on
  Automatic Control}, vol.~49, no.~9, pp.~1502--1519, 2004.

\bibitem{MoJa08}
N.~Motee and A.~Jadbabaie, ``Optimal control of spatially distributed
  systems,'' {\em IEEE Transactions on Automatic Control}, vol.~53, no.~7,
  pp.~1616--1629, 2008.

\bibitem{Ma5x}
J.~Marschak and R.~Radner, ``Economic theory of teams,'' Cowles Foundation
  Discussion Papers 59e, Cowles Foundation for Research in Economics, Yale
  University, 1958.

\bibitem{Ra62}
R.~Radner, ``Team decision problems,'' {\em The Annals of Mathematical
  Statistics}, vol.~33, no.~3, pp.~857--881, 1962.

\bibitem{Wi68}
H.~S. Witsenhausen, ``{A Counterexample in Stochastic Optimum Control},'' {\em
  SIAM Journal on Control}, vol.~6, no.~1, pp.~131--147, 1968.

\bibitem{Wi71}
H.~S. Witsenhausen, ``Separation of estimation and control for discrete time
  systems,'' {\em Proceedings of the IEEE}, vol.~59, no.~11, pp.~1557--1566,
  1971.

\bibitem{VaWa78}
P.~Varaiya and J.~Walrand, ``On delayed sharing patterns,'' {\em IEEE
  Transactions on Automatic Control}, vol.~23, no.~3, pp.~443--445, 1978.

\bibitem{NaMaTe11}
A.~Nayyar, A.~Mahajan, and D.~Teneketzis, ``Optimal control strategies in
  delayed sharing information structures,'' {\em IEEE Transactions on Automatic
  Control}, vol.~56, no.~7, pp.~1606--1620, 2011.

\bibitem{HoCh72}
Y.-C. Ho and K.-C. Chu, ``Team decision theory and information structures in
  optimal control problems -- {Part I},'' {\em IEEE Transactions on Automatic
  Control}, vol.~17, no.~1, pp.~15--22, 1972.

\bibitem{KuSi74}
B.-Z. Kurtaran and R.~Sivan, ``Linear-quadratic-gaussian control with
  one-step-delay sharing pattern,'' {\em IEEE Transactions on Automatic
  Control}, vol.~19, no.~5, pp.~571--574, 1974.

\bibitem{SaAt74}
N.~R. Sandell and M.~Athans, ``Solution of some nonclassical {{LQG}} stochastic
  decision problems,'' {\em IEEE Transactions on Automatic Control}, vol.~19,
  no.~2, pp.~108--116, 1974.

\bibitem{Yo75}
T.~Yoshikawa, ``Dynamic programming approach to decentralized stochastic
  control problems,'' {\em IEEE Transactions on Automatic Control}, vol.~20,
  no.~6, pp.~796--797, 1975.

\bibitem{MaMaRo12}
A.~Mahajan, N.~C. Martins, M.~C. Rotkowitz, and S.~Yuksel, ``Information
  structures in optimal decentralized control,'' in {\em IEEE Conference on
  Decision and Control (CDC)}, pp.~1291--1306, 2012.

\bibitem{BaVo05}
B.~Bamieh and P.~G. Voulgaris, ``A convex characterization of distributed
  control problems in spatially invariant systems with communication
  constraints,'' {\em Systems and Control Letters}, vol.~54, no.~6, pp.~575 --
  583, 2005.

\bibitem{SwLa10}
J.~Swigart and S.~Lall, ``An explicit state-space solution for a decentralized
  two-player optimal linear-quadratic regulator,'' in {\em American Control
  Conference (ACC), 2010}, pp.~6385--6390, 2010.

\bibitem{SwLa11}
J.~Swigart and S.~Lall, ``Optimal controller synthesis for a decentralized
  two-player system with partial output feedback,'' in {\em American Control
  Conference (ACC), 2011}, pp.~317--323, 2011.

\bibitem{LeNa13}
L.~Lessard and A.~Nayyar, ``Structural results and explicit solution for
  two-player {LQG} systems on a finite time horizon,'' in {\em IEEE Conference
  on Decision and Control (CDC)}, pp.~6542--6549, 2013.

\bibitem{LaLe12}
A.~Lamperski and L.~Lessard, ``Optimal state-feedback control under sparsity
  and delay constraints,'' in {\em IFAC Workshop on Distributed Estimation and
  Control in Networked Systems}, pp.~204--209, 2012.

\bibitem{LaLe14}
A.~Lamperski and L.~Lessard, ``Optimal decentralized state-feedback control
  with sparsity and delays,'' {\em Automatica}, vol.~58, pp.~143 -- 151, 2015.

\bibitem{NaLe14}
A.~Nayyar and L.~Lessard, ``Structural results for partially nested {{LQG}}
  systems over graphs,'' in {\em American Control Conference (ACC)},
  pp.~5457--5464, 2015.

\bibitem{MaLaDo14}
N.~Matni, A.~Lamperski, and J.~C. Doyle, ``Optimal two player {{LQR}} state
  feedback with varying delay,'' in {\em IFAC WC 2014}, 2014.

\bibitem{ShPa10}
P.~Shah and P.~A. Parrilo, ``H2-optimal decentralized control over posets: A
  state space solution for state-feedback,'' in {\em IEEE Conference on
  Decision and Control (CDC)}, pp.~6722--6727, 2010.

\bibitem{LaDo12}
A.~Lamperski and J.~C. Doyle, ``Dynamic programming solutions for decentralized
  state-feedback {{LQG}} problems with communication delays,'' in {\em American
  Control Conference (ACC)}, 2012.

\bibitem{LaDo13}
A.~Lamperski and J.~C. Doyle, ``Output feedback $\mathcal{H}2$ model matching
  for decentralized systems with delays,'' in {\em American Control Conference
  (ACC), 2013}, pp.~5778--5783, 2013.

\bibitem{Le12}
L.~Lessard, ``Optimal control of a fully decentralized quadratic regulator,''
  in {\em Communication, Control, and Computing (Allerton), 2012 50th Annual
  Allerton Conference on}, pp.~48--54, 2012.

\bibitem{NaKaJa14}
N.~Nayyar, D.~Kalathil, and R.~Jain, ``Optimal decentralized control in
  unidirectional one-step delayed sharing pattern with partial output
  feedback,'' in {\em American Control Conference (ACC), 2014}, pp.~1906--1911,
  2014.

\bibitem{KuVa86}
P.~R. Kumar and P.~Varaiya, {\em Stochastic Systems: Estimation, Identification
  and Adaptive Control}.
\newblock Prentice-Hall, Inc., 1986.

\bibitem{KwSi72}
H.~Kwakernaak and R.~Sivan, {\em {Linear Optimal Control Systems}}.
\newblock John Wiley \& Sons, 1972.

\bibitem{KlAt66}
D.~L. Kleinmann and M.~Athans, ``The discrete minimum principle with
  application to the linear regulator problem,'' {\em Electronic Systems
  Laboratory, MIT, Rep 260}, 1966.

\end{thebibliography}

\end{document}